\documentclass[12pt]{article}
\usepackage[top=2.54cm, bottom=2.54cm, left=2.5cm, right=2.5cm]{geometry}
\usepackage[tbtags]{amsmath}
\usepackage{amssymb}
\usepackage{amsthm}
\usepackage{fancyhdr}
\usepackage{latexsym}
\usepackage{mathrsfs}
\usepackage{xcolor}
\usepackage{wasysym}
\usepackage{fancyhdr}
\usepackage{float}
\usepackage{graphicx}
\usepackage[numbers,sort&compress]{natbib}
\usepackage{pict2e,color}
\usepackage{tikz}
\usepackage{tkz-graph}
\usepackage{enumerate}
\allowdisplaybreaks[4]

\newtheorem{theorem}{\bf Theorem}[section]

\newtheorem{claim}[theorem]{\indent Claim}

\newtheorem{proposition}[theorem]{Proposition}

\begin{document}
\title{A generalization of Stiebitz-type results on graph decomposition}
\author{Qinghou Zeng\footnote{Center for Discrete Mathematics, Fuzhou University, Fujian 350003, China. Email: zengqh@fzu.edu.cn. Research supported by Youth Foundation of Fujian Province (Grant No. JAT190021).
}~~~
Chunlei Zu\footnote{School of CyberScience and School of Mathematical Sciences, University of Science and Technology of China, Hefei, Anhui 230026, China. Email: zucle@mail.ustc.edu.cn (corresponding author). Research supported in part by National Natural Science Foundation of China grant 11622110 and Anhui Initiative in Quantum Information Technologies grant AHY150200.}}

\date{}

\maketitle

\begin{abstract}

In this paper, we consider the decomposition of multigraphs under minimum degree constraints and give a unified generalization of several results by various researchers. Let $G$ be a multigraph in which no quadrilaterals share edges with triangles and other quadrilaterals and let $\mu_G(v)=\max\{\mu_G(u,v):u\in V(G)\setminus\{v\}\}$, where $\mu_G(u,v)$ is the number of edges joining $u$ and $v$ in $G$. We show that for any two functions $a,b:V(G)\rightarrow\mathbb{N}\setminus\{0,1\}$, if $d_G(v)\ge a(v)+b(v)+2\mu_G(v)-3$ for each $v\in V(G)$, then there is a partition $(X,Y)$ of $V(G)$ such that $d_X(x)\geq a(x)$ for each $x\in X$ and $d_Y(y)\geq b(y)$ for each $y\in Y$. This extends the related results due to Diwan \cite{Diw2000}, Liu and Xu \cite{LX2017} and  Ma and Yang \cite{MY2019} on simple graphs to the multigraph setting.
\end{abstract}


\section{Introduction}
All graphs considered in this paper are finite, undirected and may have multiple edges but no loops. Let $G$ be a graph. For a subset $X\subset V(G)$, let $G[X]$ be the subgraph of $G$ induced by $X$. For each $v\in V(G)$, denote $N_X(v)$ the set of neighbors of $v$ contained in $X$ and $d_X(v)$ the number of edges between $v$ and $X\setminus\{v\}$. When $X=V(G)$, we simplify $N_{V(G)}(v)$ and $d_{V(G)}(v)$ as $N_G(v)$ and $d_G(v)$, respectively. The \emph{multiplicity} $\mu_G(u,v)$ of two different vertices $u$ and $v$ in $G$ is the number of edges joining $u$ and $v$, and the weight $\mu_G(v)$ of a vertex $v$ is defined as $\mu_G(v)=\max\left\{\mu_G(u,v):u\in V(G)\setminus\{v\}\right\}$. Call a graph $G$ \emph{simple} if $\mu_G(v)\leq1$ for each $v\in V(G)$. By a \emph{partition} $(X,Y)$ of $V(G)$, we mean that $X$, $Y$ are two disjoint nonempty sets with $X\cup Y=V(G)$. For a set $\mathscr{H}$ of graphs, we say that a graph is $\mathscr{H}$-free if it contains no member of $\mathscr{H}$ as subgraphs. We also denote $\mathbb{N}$ the set of nonnegative integers.

Many problems raised in graph theory concern graph partitioning and one popular direction of them is to partition graphs under minimum degree constraints. For a graph $G$ and two functions $a,b:V(G)\rightarrow\mathbb{N}$, a partition $(X,Y)$ of $V(G)$ is called an \emph{$(a,b)$-feasible} partition if $d_X(x)\geq a(x)$ for each $x\in X$ and $d_Y(y)\geq b(y)$ for each $y\in Y$. In 1996, Stiebitz \cite{Sti1996} proved the following celebrated result for simple graphs, solving a conjecture due to Thomassen \cite{Tho1983}.

\begin{theorem}[Stiebitz \cite{Sti1996}]\label{Stiebitz}
Let $G$ be a simple graph and $a,b:V(G)\rightarrow\mathbb{N}$ be two functions. If $d_G(x)\ge a(x)+b(x)+1$ for each $x\in V(G)$, then there is an $(a,b)$-feasible partition of $G$.
\end{theorem}

For special families of simple graphs, the minimum degree condition can be further sharpen (see \cite{Kan1998,LX2017,HMYZ2018,Diw2000,MY2019}). In particular, for $s,t\geq2$, Diwan \cite{Diw2000} showed that every simple graph with neither triangles nor quadrilaterals and minimum degree at least $s+t-1$ can already force a partition $(X,Y)$ as above. Later, Liu and Xu \cite{LX2017} generalized this result by considering triangle-free simple graphs in which no two quadrilaterals share edges.
\begin{theorem}[Liu and Xu \cite{LX2017}]\label{2C4}
Let $G$ be a triangle-free simple graph in which no two quadrilaterals share edges, and $a,b:V(G)\rightarrow\mathbb{N}\setminus\{0,1\}$ be two functions. If $d_G(x)\ge a(x)+b(x)-1$ for each $x\in V(G)$, then $G$ admits an $(a,b)$-feasible partition.
\end{theorem}
Recently, Ma and Yang \cite{MY2019} obtained the following strengthening of Diwan's result.
\begin{theorem}[Ma and Yang \cite{MY2019}]\label{C4}
Let $G$ be a quadrilateral-free simple graph and $a,b:V(G)\rightarrow\mathbb{N}\setminus\{0,1\}$ be two functions. If $d_G(x)\ge a(x)+b(x)-1$ for each $x\in V(G)$, then $G$ admits an $(a,b)$-feasible partition.
\end{theorem}

In 2017, Ban \cite{Ban2017} proved a conclusion related to Theorem \ref{Stiebitz} on weighted simple graphs. Later, Schweser and Stiebitz \cite{SS2019} further studied this problem on graphs, and generalized the results of Stiebitz \cite{Sti1996} and Liu and Xu \cite{LX2017} from simple graphs to graphs. Very recently, confirming two conjectures of Schweser and Stiebitz, Liu and Xu \cite{LX2019} obtained a graph version of Theorem \ref{2C4}.
\begin{theorem}[Liu and Xu \cite{LX2019}]\label{multi-2C4}
Let $G$ be a triangle-free graph in which no two quadrilaterals share edges, and $a,b:V(G)\rightarrow\mathbb{N}\setminus\{0,1\}$ be two functions. If $d_G(x)\ge a(x)+b(x)+2\mu_G(x)-3$ for each $x\in V(G)$, then $G$ admits an $(a,b)$-feasible partition.
\end{theorem}

For related problems on graph partitioning under degree constraints or other variances, we refer readers to \cite{BTV2007,FLW2019,Lov1966,She1988,SX2020}. In this paper, we consider the partitions of graphs and give a unified generalization of Theorems \ref{2C4}, \ref{C4} and \ref{multi-2C4} as well as the result of Diwan \cite{Diw2000}. Precisely, we establish the following theorem.

\begin{theorem}\label{multigraph}
Let $G$ be a graph in which no quadrilaterals share edges with triangles and other quadrilaterals, and let $a,b:V(G)\rightarrow\mathbb{N}\setminus\{0,1\}$ be two functions. If $d_G(x)\ge a(x)+b(x)+2\mu_G(x)-3$ for each $x\in V(G)$, then $G$ admits an $(a,b)$-feasible partition.
\end{theorem}

Note that this is tight for cycles in the following two perspectives. Firstly, the ranges of the functions $a,b$ cannot be relaxed to the set of integers at least one by choosing the constant functions $a=b-1=1$. Secondly, one also cannot lower the degree condition further by choosing the constant functions $a=b=2$. We also mention that $G$ is actually $\{K_4^-,C_5^+,K_{2,3},L_3\}$-free in Theorem \ref{multigraph}, where $K_4^-$ is the graph obtained from $K_4$ by removing one edge, $C_5^+$ is the graph obtained from $C_5$ by adding one edge between two nonadjacent vertices, and $L_3$ is the graph consisting of two quadrilaterals sharing exactly one common edge. Additionally, we use the condition that $G$ is $L_3$-free exactly once (see Claim \ref{min}) in our proof; however, this condition is necessary as shown by the graph constructed in \cite{Zu2020}.


\section{Notations and Propositions}\label{Notation}
Let $G$ be a graph and $f:V(G)\rightarrow\mathbb{N}\setminus\{0,1\}$ be a function. For a subset $X\subseteq V(G)$, we say that (i) $X$ is \emph{$f$-nice} if $d_X(x)\geq f(x)+\mu_G(x)-1$ for each $x\in X$, (ii) $X$ is \emph{$f$-feasible} if $d_X(x)\geq f(x)$ for each $x\in X$, (iii) $X$ is \emph{$f$-meager} if for each nonempty subset $X'\subseteq X$ there exists a vertex $x\in X'$ such that $d_{X'}(x)\leq f(x)+\mu_G(x)-1$, and (iv) $X$ is \emph{$f$-degenerate} if for each nonempty subset $X'\subseteq X$ there exists a vertex $x\in X'$ such that $d_{X'}(x)\leq f(x)$. We have the following propositions immediately from the definitions.

\begin{proposition}\label{pro2}
If $\mu_G(x)\geq1$ for each $x\in V(G)$, then each $f$-nice subset is also $f$-feasible and each $f$-degenerate subset is also $f$-meager.
\end{proposition}

\begin{proposition}\label{pro1}
A subset of $V(G)$ does not contain any $f$-feasible subset if and only if it is $(f-1)$-degenerate.
\end{proposition}



For a graph $G$ and two functions $a,b:V(G)\rightarrow\mathbb{N}$, a pair $(X,Y)$ of disjoint subsets of $V(G)$ is called an \emph{$(a,b)$-feasible} pair if $X$ is $a$-feasible and $Y$ is $b$-feasible; if in addition $(X,Y)$ is a partition of $V(G)$, then we call it an $(a,b)$-feasible partition. Similarly, a partition $(X,Y)$ of $V(G)$ is called an \emph{$(a,b)$-meager} partition if $X$ is $a$-meager and $Y$ is $b$-meager. The following proposition due to Schweser and Stiebitz \cite{SS2019} plays a vital role in our proof of Theorem \ref{multigraph}.

\begin{proposition}[Schweser and Stiebitz \cite{SS2019}]\label{key}
Let $G$ be a graph without isolated vertices, and let $a,b:V(G)\rightarrow\mathbb{N}$ be two functions such that $d_G(x)\ge a(x)+b(x)+2\mu_G(x)-3$ for each $x\in V(G)$. If $G$ has an $(a,b)$-feasible pair, then it admits an $(a,b)$-feasible partition.
\end{proposition}

Let $G$ be a graph and let $a,b:V(G)\rightarrow\mathbb{N}$ be two functions. For each partition $(A,B)$ of $V(G)$, we define the weight $\omega(A,B)$ of $(A,B)$ as
\[
\omega(A,B)=|E(G[A])|+|E(G[B])|+\sum_{u\in A}b(u)+\sum_{v\in B}a(v).
\]
Then, for each $u\in A$ and $v\in B$,  simple calculations show that
\begin{eqnarray}\label{weight1}
\omega(A\setminus\{u\},B\cup\{u\})-\omega(A,B)=d_B(u)-d_A(u)+a(u)-b(u),
\end{eqnarray}
\begin{eqnarray}\label{weight2}
\omega(A\cup\{v\},B\setminus\{v\})-\omega(A,B)=d_A(v)-d_B(v)+b(v)-a(v),
\end{eqnarray}
and
\begin{align}\label{weight3}
&\;\quad\omega(A\cup\{v\}\backslash\{u\},B\cup\{u\}\backslash\{v\})-\omega(A,B)\notag
\\&=d_B(u)-d_A(u)+a(u)-b(u)+d_A(v)-d_B(v)+b(v)-a(v)-2\mu_G(u,v).
\end{align}

\section{Proof of Theorem \ref{multigraph}}\label{Main-Result}
Throughout this section, let $G$ be a $\{K_4^-,C_5^+,K_{2,3},L_3\}$-free graph and $a,b:V(G)\rightarrow\mathbb{N}\setminus\{0,1\}$ be two functions such that $d_G(x)\ge a(x)+b(x)+2\mu_G(x)-3$ for each $x\in V(G)$. Clearly, $d_G(x)\geq1$ for each $x\in V(G)$. Thus, $\mu_G(x)\geq1$ for each $x\in V(G)$. Since there is no danger of confusion, the reference to $G$ in the subscript of $\mu_G$ will be dropped in the following proof.

Suppose for a contradiction that $G$ contains no $(a,b)$-feasible partitions. It follows from Proposition \ref{key} that there is no $(a,b)$-feasible pair in $G$. We may assume that
\begin{eqnarray}\label{=}
d_G(x)=a(x)+b(x)+2\mu(x)-3
\end{eqnarray}
for each $x\in V(G)$. Otherwise, we can increase $a,b$ to get functions $a',b'$ such that $a'\ge a$, $b'\ge b$ and $d_G(x)=a'(x)+b'(x)+2\mu(x)-3$ for each $x\in V(G)$. Clearly, the existence of an $(a',b')$-feasible partition would guarantee that of an $(a,b)$-feasible partition in $G$.

\begin{claim}\label{(a-1,b-1)-meager}
There exists an $(a-1,b-1)$-meager partition in $G$.
\end{claim}
\begin{proof}
Observe that there is an $a$-nice proper subset of $V(G)$. Indeed, for a fixed $u\in V(G)$ and each $x\in V(G)\setminus\{u\}$, it follows from \eqref{=} that
\[
d_{V(G)\setminus\{u\}}(x)=d_G(x)-\mu(u,x)\geq a(x)+b(x)+\mu(x)-3\geq a(x)+\mu(x)-1,
\]
meaning that $V(G)\setminus\{u\}$ is $a$-nice. Let $S$ be a minimum $a$-nice subset of $V(G)$ and $T=V(G)\setminus S$. Clearly, $|S|\ge2$ and $T\neq\emptyset$. Note that $S$ is $a$-feasible by Proposition \ref{pro2}. Since $G$ has no $(a,b)$-feasible pair, $T$ contains no $b$-feasible subset. By Proposition \ref{pro1}, $T$ is $(b-1)$-degenerate, and thus is $(b-1)$-meager. Take $v\in S$ and it follows that $S\setminus\{v\}$ is $(a-1)$-meager by the minimality of $S$. Note that $d_S(v)\geq a(v)+\mu(v)-1$. This together with \eqref{=} yields that $d_{T\cup\{v\}}(v)=d_T(v)\leq b(v)+\mu(v)-2$. Thus, $T\cup\{v\}$ is $(b-1)$-meager. If not, then there is a $b$-nice subset $T'\subseteq T\cup\{v\}$. Since $T$ is $(b-1)$-meager, we have $v\in T'$ and $d_{T\cup\{v\}}(v)\geq d_{T'}(v)\geq b(v)+\mu(v)-1$, a contradiction. Consequently, $(S\setminus\{v\},T\cup\{v\})$ is an $(a-1,b-1)$-meager partition in $G$, as desired.
\end{proof}

Let $\mathscr{P}$ be the family of all $(a-1,b-1)$-meager partitions $(A,B)$ satisfying that $\omega(A,B)$ is maximum. For any $(A,B)\in\mathscr{P}$, let $A^-=\{u\in A\mid d_A(u)\leq a(u)+\mu(u)-2\}$ and $B^-=\{v\in B\mid d_B(v)\leq b(v)+\mu(v)-2\}$. Note that both $A^-$ and $B^-$ are nonempty by the definition of $\mathscr{P}$. So for any $v\in B^-$, $d_A(v)=d_G(v)-d_B(v)\geq a(v)+\mu(v)-1$, implying $|A|\geq2$. Similarly, $|B|\geq2$.

\begin{claim}\label{containment}
For any $(A,B)\in\mathscr{P}$, $u\in A^-$ and $v\in B^-$, we have $A\cup\{v\}$ is not $(a-1)$-meager and every $a$-nice subset of $A\cup\{v\}$ contains $u$ and $v$; furthermore, $B\cup\{u\}$ is not $(b-1)$-meager and every $b$-nice subset of $B\cup\{u\}$ contains $u$ and $v$.
\end{claim}
\begin{proof}
Note that $\omega(A\cup\{v\},B\setminus\{v\})-\omega(A,B)=d_G(v)-2d_B(v)+b(v)-a(v)$ by \eqref{weight2}. This together with \eqref{=} and $d_B(v)\le b(v)+\mu(v)-2$ implies that
$
\omega(A\cup\{v\},B\setminus\{v\})-\omega(A,B)\geq1.
$
Thus $(A\cup\{v\},B\setminus\{v\})$ cannot be an $(a-1,b-1)$-meager partition by the maximality of $\omega(A,B)$. Since $B\setminus\{v\}$ is $(b-1)$-meager, $A\cup\{v\}$ cannot be $(a-1)$-meager. Similarly, $B\cup\{u\}$ is not $(b-1)$-meager. Hence there exist an $a$-nice subset $A'\subseteq A\cup\{v\}$ and a $b$-nice subset $B'\subseteq B\cup\{u\}$. Since $A$ is $(a-1)$-meager and $B$ is $(b-1)$-meager, we have $v\in A'$ and $u\in B'$. Now, we prove that $u\in A'$ and $v\in B'$. If $u\notin A'$ and $v\notin B'$, then $(A',B')$ is an $(a,b)$-feasible pair by Proposition \ref{pro2}, a contradiction. Suppose by symmetry that $u\in A'$ and $v\notin B'$. Clearly, $B'\subseteq (B\cup\{u\})\setminus\{v\}$ and $d_{B\setminus\{v\}}(u)=d_{B\cup\{u\}\setminus\{v\}}(u)\geq d_{B'}(u) \geq b(u)+\mu(u)-1$. Thus, $d_{A'}(u)\leq d_{A\cup\{v\}}(u)=d_{G}(u)-d_{B\setminus\{v\}}(u)\leq a(u)+\mu(u)-2$, a contradiction.
\end{proof}

Let $A^\ast\subseteq A$ such that $A^\ast\cap A^-\neq\emptyset$. By Claim \ref{containment}, $B\cup A^\ast$ is not $(b-1)$-meager and there exists a $b$-nice subset of $B\cup A^\ast$, indicating that $A\setminus A^\ast$ is $(a-1)$-degenerate as $G$ has no $(a,b)$-feasible pair. Similarly, if $B^\ast\subseteq B$ such that $B^\ast\cap B^-\neq\emptyset$, then $B\setminus B^\ast$ is $(b-1)$-degenerate. We point out that Claim \ref{containment} will be also used in this form frequently.

\begin{claim}\label{complete}
For any $(A,B)\in\mathscr{P}$, every vertex in $A^-$ is adjacent to every vertex in $B^-$.
\end{claim}
\begin{proof}
Suppose that there exist $u\in A^-$ and $v\in B^-$ such that $\mu(u,v)=0$. By Claim \ref{containment}, there is an $a$-nice subset $A'\subseteq A\cup\{v\}$ such that $u\in A'$, implying that $d_{A'}(u)\geq a(u)+\mu(u)-1$. However, $d_{A'}(u)\leq d_{A\cup\{v\}}(u)=d_A(u)+\mu(u,v)\leq a(u)+\mu(u)-2$, a contradiction.
\end{proof}

Recall that both $A^-$ and $B^-$ are nonempty. By Claim \ref{complete}, either $|A^-|=|B^-|=2$ or $\min\{|A^-|,|B^-|\}=1$ as $G$ is $K_{2,3}$-free.

\begin{claim}\label{nonempty}
For any $(A,B)\in\mathscr{P}$, we have $A\setminus A^-\neq\emptyset$ and $B\setminus B^-\neq\emptyset$.
\end{claim}
\begin{proof}
For each $u\in A^-$, there exists a $b$-nice subset $B'\subseteq B\cup\{u\}$ by Claim \ref{containment}. It follows that $d_{B'}(y)\geq b(y)+\mu(y)-1\geq b(y)$ for each $y\in B'$, implying $|N_{B'}(y)|\geq2$. If $|A^-|=|B^-|=2$, then we let $B^-=\{v_1,v_2\}$. Since $G$ is $K_4^-$-free, $v_1v_2\notin E(G)$ by Claim \ref{complete}. Thus, $N_{B'}(v_1)=N_{B'}(v_2)=\{u\}$ providing $B=B^-$. This leads to a contradiction as $v_i\in B'$ for some $i=1,2$, implying $B\setminus B^-\neq\emptyset$. Similarly, $A\setminus A^-\neq\emptyset$. If $\min\{|A^-|,|B^-|\}=1$, then we assume that $A^-=\{u\}$. Clearly, $A\setminus A^-\neq\emptyset$ as $|A|\geq2$. Since $A$ is $(a-1)$-meager, there exists $x\in A\setminus\{u\}$ such that $d_{A\setminus\{u\}}(x)\leq a(x)+\mu(x)-2$. Note that $d_{A\setminus\{u\}}(x)+\mu(u,x)=d_A(x)\geq a(x)+\mu(x)-1$. It follows that $\mu(u,x)\geq1$ and $d_A(x)\leq a(x)+2\mu(x)-2$, yielding that $ux\in E(G)$ and $d_B(x)=d_G(x)-d_A(x)\geq b(x)-1\geq1$. Suppose that $B=B^-$ and $z\in N_B(x)$. Choose $v=z$ in Claim \ref{containment}, implying $z\in B'$. Since $|N_{B'}(z)|\geq2$, there exists $z'\in B^-\setminus\{z\}$ such that $zz'\in E(G)$. By Claim \ref{complete}, $\{u,x,z,z'\}$ forms a $K_4^-$, a contradiction. Thus $B\setminus B^-\neq\emptyset$.
\end{proof}

For any $(A,B)\in\mathscr{P}$, let $D_A=\{u\in A\mid d_A(u)\leq a(u)-1\}$ and $D_B=\{v\in B\mid d_B(v)\leq b(v)-1\}$. Clearly, $D_A\subseteq A^-$ and $D_B\subseteq B^-$.
\begin{claim}\label{degree}
For any $(A,B)\in\mathscr{P}$, $u\in A^-$ and $v\in B^-$, if either $u\in D_A$ or $v\in D_B$, then $(A\cup\{v\}\setminus\{u\},B\cup\{u\}\setminus\{v\})\in\mathscr{P}$. Moreover, if $u\in D_A$, then $\mu(u,v)=\mu(u)$, $d_A(u)=a(u)-1$ and $d_B(v)=b(v)+\mu(v)-2$;  if $v\in D_B$, then $\mu(u,v)=\mu(v)$, $d_B(v)=b(v)-1$ and $d_A(u)=a(u)+\mu(u)-2$.
\end{claim}
\begin{proof}
Since every $a$-nice subset of $A\cup\{v\}$ contains $u$ by Claim \ref{containment}, $A\cup\{v\}\setminus\{u\}$ is $(a-1)$-meager. Similarly, $B\cup\{u\}\setminus\{v\}$ is $(b-1)$-meager. Thus $(A\cup\{v\}\setminus\{u\},B\cup\{u\}\setminus\{v\})$ is an $(a-1,b-1)$-meager partition. By \eqref{weight3}, 
$\omega(A\cup\{v\}\backslash\{u\},B\cup\{u\}\backslash\{v\})-\omega(A,B)=(d_G(u)-2d_A(u)+a(u)-b(u))+(d_G(v)-2d_B(v)+b(v)-a(v))-2\mu(u,v)$.
Suppose by symmetry that $u\in D_A$. Since $d_A(u)\leq a(u)-1$ and $d_B(v)\leq b(v)+\mu(v)-2$, by \eqref{=}, we have
$$\omega(A\cup\{v\}\backslash\{u\},B\cup\{u\}\backslash\{v\})-\omega(A,B)\geq(2\mu(u)-1)+1-2\mu(u,v)=2(\mu(u)-\mu(u,v))\geq0.$$
By the maximality of $\omega(A,B)$, $\omega(A\cup\{v\}\backslash\{u\},B\cup\{u\}\backslash\{v\})=\omega(A,B)$. Thus $(A\cup\{v\}\setminus\{u\},B\cup\{u\}\setminus\{v\})\in\mathscr{P}$, $\mu(u,v)=\mu(u)$, $d_A(u)=a(u)-1$ and $d_B(v)=b(v)+\mu(v)-2$.
\end{proof}

By Claim \ref{degree}, $D_A=\{u\in A\mid d_A(u)=a(u)-1\}$ and $D_B=\{v\in B\mid d_B(v)=b(v)-1\}$; in addition, $d_A(u)\geq a(u)-1$ and $d_B(v)\geq b(v)-1$ for each $u\in A$ and $v\in B$.

\begin{claim}\label{min}
For any $(A,B)\in\mathscr{P}$, we have $\min\{|A^-|,|B^-|\}=1$.
\end{claim}
\begin{proof}
Suppose for a contradiction that $A^-=\{u_1,u_2\}$ and $B^-=\{v_1,v_2\}$. Since $G$ is $K_4^-$-free, $u_1u_2,v_1v_2\notin E(G)$ by Claim \ref{complete}. Note that $A\cup B^-$ is not $(a-1)$-meager by Claim \ref{containment}. It follows that $B\setminus B^-$ is $(b-1)$-degenerate as $G$ has no $(a,b)$-feasible pair and $B\setminus B^-\neq\emptyset$ by Claim \ref{nonempty}. Thus there exists $y\in B\setminus B^-$ such that $d_{B\setminus B^-}(y)\leq b(y)-1$, implying $N_{B^-}(y)\neq\emptyset$ as $d_B(y)\geq b(y)+\mu(y)-1\geq b(y)$. By Claim \ref{complete}, $|N_{B^-}(y)|=1$ as $G$ is $K_{2,3}$-free, say $N_{B^-}(y)=\{v_1\}$. By symmetry, $A\setminus A^-$ is $(a-1)$-degenerate and there exists $x_1\in A\setminus A^-$ such that $d_{A\setminus A^-}(x_1)\leq a(x_1)-1$ and $|N_{A^-}(x_1)|=1$, say $N_{A^-}(x_1)=\{u_1\}$.  Clearly, $d_{A\setminus \{u_1\}}(x_1)=d_{A\setminus A^-}(x_1)\leq a(x_1)-1$ and $d_{B\setminus \{v_1\}}(y)=d_{B\setminus B^-}(y)\leq b(y)-1$.

Since $G$ has no $(a,b)$-feasible partition, either $A$ is $(a-1)$-degenerate or $B$ is $(b-1)$-degenerate. We may assume that $A$ is $(a-1)$-degenerate. Thus either  $d_A(u_1)\leq a(u_1)-1$ or $d_A(u_2)\leq a(u_2)-1$. If $d_A(u_1)\leq a(u_1)-1$, then we set $u:=u_1$ and $x:=x_1$. If $d_A(u_1)\geq a(u_1)$, then $d_A(u_2)\leq a(u_2)-1$. Clearly, $A\setminus\{u_2\}$ is $(a-1)$-degenerate. Thus there exists $x_2\in A\setminus\{u_2\}$ such that $d_{A\setminus\{u_2\}}(x_2)\leq a(x_2)-1$. Note that $d_{A\setminus\{u_2\}}(u_1)=d_A(u_1)\geq a(u_1)$ as $u_1u_2\notin E(G)$. Thus $x_2\neq u_1$ and $x_2\in A\setminus A^-$. Note also that $d_A(x_2)\geq a(x_2)+\mu(x_2)-1\geq a(x_2)$. This implies $u_2x_2\in E(G)$. Set $u:=u_2$ and $x:=x_2$. In both cases, we have $ux\in E(G)$, $d_A(u)\leq a(u)-1$ and $d_{A\setminus\{u\}}(x)\leq a(x)-1$. Since $G$ is $C_5^+$-free, we have $xv_1,uy\notin E(G)$. By Claim \ref{degree}, $(A_0,B_0):=(A\cup\{v_1\}\setminus\{u\},B\cup\{u\}\setminus\{v_1\})\in\mathscr{P}$. Observe that $d_{A_0}(x)=d_{A\setminus\{u\}}(x)\leq a(x)-1$ and $d_{B_0}(y)=d_{B\setminus \{v_1\}}(y)\leq b(y)-1$. Thus $x\in A_0^-$ and $y\in B_0^-$, yielding $xy\in E(G)$ by Claim \ref{complete}. It follows that $\{u_1,u_2,v_1,v_2,x,y\}$ contains an $L_3$, a contradiction.
\end{proof}

For any $(A,B)\in\mathscr{P}$, define $A^==\{x\in A\mid d_A(x)= a(x)+\mu(x)-1\}$ and $B^==\{y\in B\mid d_B(y)=b(y)+\mu(y)-1\}$. A path $x u v y$ is called a \emph{special path} with respect to $(A,B)$, if $u\in A^-$, $v\in B^-$, $x\in A^=$ and $y\in B^=$.

\begin{claim}\label{diagonal}
For any special path $x u v y$ with respect to $(A,B)\in\mathscr{P}$, if either $u\in D_A$ or $v\in D_B$, then either $vx\in E(G)$ or $uy\in E(G)$. Moreover, if $vx\in E(G)$, then $N_{A^=}(u)=\{x\}$; if $uy\in E(G)$, then $N_{B^=}(v)=\{y\}$.
\end{claim}
\begin{proof}
Suppose that $vx,uy\notin E(G)$. We may assume by symmetry that $u\in D_A$. By Claim \ref{degree}, $(A_1,B_1):=(A\cup\{v\}\setminus\{u\},B\cup\{u\}\setminus\{v\})\in\mathscr{P}$, $\mu(u,v)=\mu(u)$, $d_A(u)=a(u)-1$ and $d_B(v)=b(v)+\mu(v)-2$. This together with $d_{A_1}(v)=d_G(v)-d_B(v)-\mu(u,v)$ and  $d_{B_1}(u)=d_G(u)-d_A(u)-\mu(u,v)$ implies $v\in A_1^-$ and $u\in B_1^-$. Since $x\in A^=$ and $y\in B^=$, we have $d_{A_1}(x)=d_A(x)-\mu(u,x)=a(x)+\mu(x)-1-\mu(u,x)$ and $d_{B_1}(y)=d_B(y)-\mu(v,y)=b(y)+\mu(y)-1-\mu(v,y)$, indicating $x\in A_1^-$ and $y\in B_1^-$. This contradicts Claim \ref{min}.

Suppose that $vx\in E(G)$ and there exists $x'\in N_{A^=}(u)\setminus\{x\}$. Clearly, $x' u v y$ forms another special path with respect to $(A,B)$. It follows that either $uy\in E(G)$ or $vx'\in E(G)$. In both cases, we can find a $K_4^-$, a contradiction. Similarly, if $uy\in E(G)$, then $N_{B^=}(v)=\{y\}$.
\end{proof}

\begin{claim}\label{exchange}
For any $(A,B)\in\mathscr{P}$, let $u\in A^-$ and $v\in B^-$. If $u\in D_A$ and $x\in N_{A^=}(u)$ with $vx\notin E(G)$, then  $(A\cup\{v\}\setminus\{x\},B\cup\{x\}\setminus\{v\})\in\mathscr{P}$; if $v\in D_B$ and $y\in N_{B^=}(v)$ with $uy\notin E(G)$, then  $(A\cup\{y\}\setminus\{u\},B\cup\{u\}\setminus\{y\})\in\mathscr{P}$.
\end{claim}
\begin{proof}
Assume that $u\in D_A$ and $x\in N_{A^=}(u)$ with $vx\notin E(G)$.
We first show that $B\cup\{x\}\setminus\{v\}$ is $(b-1)$-meager. If not, then there is a $b$-nice subset $B'\subseteq B\cup\{x\}\setminus\{v\}$. This implies that $x\in B'$ as $B$ is $(b-1)$-meager. Since $vx\notin E(G)$ and $x\in A^=$, $d_{B'}(x)\leq d_{B\cup\{x\}\setminus\{v\}}(x)=d_B(x)=d_G(x)-d_A(x)=b(x)+\mu(x)-2$, contradicting with $x\in B'$. Now, we prove that $A\cup\{v\}\setminus\{x\}$ is $(a-1)$-meager. Otherwise, there is an $a$-nice subset $A'\subseteq A\cup\{v\}\setminus\{x\}$. Since $A$ is $(a-1)$-meager, we have $v\in A'$ and $d_{A'}(v)\geq a(v)+\mu(v)-1$. Note that $d_B(v)=b(v)+\mu(v)-2$ by Claim \ref{degree} as $u\in D_A$. It follows that $d_{A'}(v)\leq d_{A\cup\{v\}\setminus\{x\}}(v)=d_A(v)=a(v)+\mu(v)-1$ as $vx\notin E(G)$. Thus, $d_{A'}(v)=d_A(v)$, implying $u\in A'$ as $uv\in E(G)$. The fact $d_{A'}(u)\leq d_{A\cup\{v\}\setminus\{x\}}(u)=d_A(u)+\mu(u,v)-\mu(u,x)\leq a(u)+\mu(u)-2$ also indicates that $u\notin A'$, a contradiction. Therefore, $(A\cup\{v\}\setminus\{x\},B\cup\{x\}\setminus\{v\})$ is an $(a-1,b-1)$-meager partition. With simple calculations, we have $\omega((A\cup\{v\}\setminus\{x\},B\cup\{x\}\setminus\{v\}))=\omega(A,B)$ in view of \eqref{weight3} and \eqref{=}.
Thus $(A\cup\{v\}\setminus\{x\},B\cup\{x\}\setminus\{v\})\in\mathscr{P}$. Similarly, if $v\in D_B$ and $y\in N_{B^=}(v)$ with $uy\notin E(G)$, then  $(A\cup\{y\}\setminus\{u\},B\cup\{u\}\setminus\{y\})\in\mathscr{P}$.
\end{proof}


Fix a partition $(A,B)\in\mathscr{P}$. By Claim \ref{min}, we may assume by symmetry that
\[
A^-=\{u\}\;\, \text{and}\;\, |B^-|\geq|A^-|.
\]
By Claim \ref{containment}, $B\cup\{u\}$ is not $(b-1)$-meager. Since $G$ has no $(a,b)$-feasible pair, $A\setminus\{u\}$ is $(a-1)$-degenerate, implying that there exists $x_1\in A\setminus\{u\}$ such that $d_{A\setminus\{u\}}(x_1)\leq a(x_1)-1$. Note that $d_A(x_1)\geq a(x_1)+\mu(x_1)-1$ as $x_1\in A\setminus A^-$ and $d_{A\setminus\{u\}}(x_1)=d_A(x_1)-\mu(u,x_1)$. It follows that $\mu(u,x_1)=\mu(x_1)$, $d_{A\setminus\{u\}}(x_1)=a(x_1)-1$ and $d_A(x_1)=a(x_1)+\mu(x_1)-1$. Hence, 
\[
x_1\in N_{A^=}(u).
\]

Recall that either $A$ is $(a-1)$-degenerate or $B$ is $(b-1)$-degenerate. It follows that either $D_A\neq\emptyset$ or $D_B\neq\emptyset$. In what follows, we may assume that
\begin{eqnarray}\label{Nonempty}
D_B\neq\emptyset.
\end{eqnarray}
Otherwise, let $D_B=\emptyset$. Clearly, $B$ is $b$-feasible and $A$ is $(a-1)$-degenerate. Thus $D_A=\{u\}$.
If $|B^-|=1$, then the case can be reduced to \eqref{Nonempty} by symmetry as $D_A\neq\emptyset$.
Suppose that $|B^-|\geq2$ and $v_1,v_2\in B^-$.
Since $G$ is $K_4^-$-free, either $x_1v_1\notin E(G)$ or $x_1v_2\notin E(G)$ by Claim \ref{complete}. By symmetry, assume that $x_1v_1\notin E(G)$.
Clearly, $(A_2,B_2):=(A\cup\{v_1\}\setminus\{u\},B\cup\{u\}\setminus\{v_1\})\in\mathscr{P}$, $\mu(u,v)=\mu(u)$ and $d_B(v)=b(v)+\mu(v)-2$ for each $v\in B^-$ by Claim \ref{degree}.
It is easy to check that $v_1\in A_2^-$, $x_1\in D_{A_2}\subseteq A_2^-$ and $u\in B_2^-$. Thus $B_2^-=\{u\}$ by Claim \ref{min}. Again, this can be reduced to \eqref{Nonempty} as $|B_2^-|=1$ and $D_{A_2}\neq\emptyset$ .

For each $v\in D_B$ and the fixed vertex $x_1$, let $A_v=A\cup\{v\}\setminus\{x_1\}$ and $B_v=B\cup\{x_1\}\setminus\{v\}$.
\begin{claim}\label{exchange-x1v}
For each $v\in D_B$, if $x_1v\notin E(G)$, then $(\mathrm{i})$ $\mu(v)=1$; $(\mathrm{ii})$ $(A_v, B_v)\in\mathscr{P}$, $u\in A_v^-$, $v\in A_v^=$ and $x_1\in B_v^-$.
\end{claim}
\begin{proof}
$(\mathrm{i})$ By Claim \ref{degree}, $(A_3,B_3):=(A\cup\{v\}\setminus\{u\},B\cup\{u\}\setminus\{v\})\in\mathscr{P}$, $\mu(v)=\mu(u,v)$ and $d_A(u)=a(u)+\mu(u)-2$ as $v\in D_B$.
Recall that $d_{A\setminus\{u\}}(x_1)=a(x_1)-1$. Thus $d_{A_3}(x_1)=d_{A\setminus\{u\}}(x_1)=a(x_1)-1$ as $x_1v\notin E(G)$, yielding $x_1\in D_{A_3}$. Note that $d_{B_3}(u)=d_G(u)-d_A(u)-\mu(u,v)=b(u)+\mu(u)-1-\mu(u,v)$. This implies $u\in B_3^-$ as $\mu(u,v)\geq1$. Applying Claim \ref{degree} with $(A_3,B_3)\in\mathscr{P}$, $x_1\in D_{A_3}$ and $u\in B_3^-$, we have $d_{B_3}(u)=b(u)+\mu(u)-2$. It follows that $\mu(u,v)=1$, implying $\mu(v)=1$.

$(\mathrm{ii})$ Recall that $d_A(u)=a(u)+\mu(u)-2$ and $\mu(u,v)=\mu(v)=1$. Since $v\in D_B$ and $x_1\in A^=$, we have $d_{A_v}(u)=d_A(u)+\mu(u,v)-\mu(u,x_1)=a(u)+\mu(u)-1-\mu(u,x_1)$, $d_{A_v}(v)=d_G(v)-d_B(v)=a(v)$ and $d_{B_v}(x_1)=d_G(x_1)-d_A(x_1)=b(x_1)+\mu(x_1)-2$. Now, we show that $B_v$ is $(b-1)$-meager. If not, then there exists a $b$-nice subset $B'\subseteq B_v$. Since $B$ is $(b-1)$-meager, we have $x_1\in B'$ and $d_{B_v}(x_1)\geq d_{B'}(x_1)\geq b(x_1)+\mu(x_1)-1$, a contradiction. Next, we prove that $A_v$ is $(a-1)$-meager. Otherwise, there is an $a$-nice subset $A'\subseteq A_v$. Since $A$ is $(a-1)$-meager, we have $v\in A'$ and $d_{A_v}(v)\geq d_{A'}(v)\geq a(v)+\mu(v)-1=a(v)$. This implies that $d_{A_v}(v)=d_{A'}(v)$. Thus $u\in A'$ as $uv\in E(G)$. It follows that $d_{A_v}(u)\geq d_{A'}(u)\geq a(u)+\mu(u)-1$, a contradiction. Therefore, $(A_v,B_v)$ is an $(a-1,b-1)$-meager partition. Simple calculations together with \eqref{weight3} and \eqref{=} show that $\omega(A_v,B_v)=\omega(A,B)$, implying $(A_v,B_v)\in\mathscr{P}$. Moreover, $u\in A_v^-$, $v\in A_v^=$ and $x_1\in B_v^-$ by noting that $\mu(u,x_1)\ge1$ and $\mu(v)=1$.
\end{proof}



Now, we conclude that $D_B$ is an independent set. Otherwise, there is an edge $vv'$ contained in $G[D_B]$. Since $G$ is $K_4^-$-free, we have $x_1v,x_1v'\notin E(G)$. By Claim \ref{exchange-x1v}, $\mu(v)=1$ and $(A_v, B_v)\in\mathscr{P}$. It follows that $d_{B_v}(v')=d_{B}(v')-\mu(v,v')=b(v')-2$, contradicting Claim \ref{degree}. 

Note that $B\setminus D_B$ is $(b-1)$-degenerate by Claim \ref{containment} as $B\setminus D_B\neq\emptyset$ by Claim \ref{nonempty}. Thus, there exists $y\in B\setminus D_B$ such that $d_{B\setminus D_B}(y)\leq b(y)-1$.

\begin{claim}\label{y}
For each $y\in B\setminus D_B$ satisfying $d_{B\setminus D_B}(y)\leq b(y)-1$, we have $|N_{D_B}(y)|=1$.
\end{claim}
\begin{proof}
Note that $d_B(y)=d_{B\setminus D_B}(y)+d_{D_B}(y)\geq b(y)$ as $y\in B\setminus D_B$. It follows that $d_{D_B}(y)\geq 1$. This together with Claim \ref{complete} yields that $1\leq|N_{D_B}(y)|\leq2$ as $G$ is $K_{2,3}$-free. Suppose that $N_{D_B}(y)=\{v_1,v_2\}$ and $v_1v_2\notin E(G)$ as $D_B$ is independent. Clearly,
$d_B(y)=d_{B\setminus D_B}(y)+d_{D_B}(y)\leq b(y)-1+\mu(v_1,y)+\mu(v_2,y).$
Since $G$ is $\{C_5^+,K_{2,3}\}$-free, $x_1v_1,x_1v_2,x_1y\notin E(G)$. By Claim \ref{exchange-x1v}, $(A_{v_1},B_{v_1})\in\mathscr{P}$, $u\in A_{v_1}^-$ and $v_1\in A_{v_1}^=$. Note also that $v_2\in D_{B_{v_1}}$ as $d_{B_{v_1}}(v_2)=d_B(v_2)=b(v_2)-1$. Since $d_{B_{v_1}}(y)=d_B(y)-\mu(v_1,y)\leq b(y)-1+\mu(v_2,y)\leq b(y)+\mu(y)-1$, we have either $y\in B_{v_1}^-$ or $y\in B_{v_1}^=$. If $y\in B_{v_1}^-$, then $uy\in E(G)$ by Claim \ref{complete}; if $y\in B_{v_1}^=$, then $v_1 u v_2 y$ forms a special path with respect to $(A_{v_1},B_{v_1})$, indicating that either $uy\in E(G)$ or $v_1v_2\in E(G)$ by Claim \ref{diagonal}. In both cases, $\{u,v_1,v_2,y\}$ contains a $K_4^-$, a contradiction.
\end{proof}

By Claim \ref{y}, we can fix such a vertex $y\in B\setminus D_B$ and assume that
\[
N_{D_B}(y)=\{v_1\}
\]
for some vertex $v_1\in D_B$. It follows that $d_B(y)=d_{B\setminus D_B}(y)+d_{D_B}(y)\leq b(y)-1+\mu(v_1,y)\leq b(y)+\mu(y)-1,$ thus either $y\in B^-\setminus D_B$ or $y\in B^=$. If $y\in B^-\setminus D_B$, then $uy\in E(G)$ by Claim \ref{complete}. If $y\in B^=$, then $x_1 u v_1 y$ forms a special path with respect to $(A,B)$. Since $v_1\in D_B$, we have either $x_1v_1\in E(G)$ or $uy\in E(G)$ by Claim \ref{diagonal}. Hence, we conclude
\begin{eqnarray}\label{x1v1uy}
\text{either} \; x_1v_1\in E(G) \; \text{or} \; uy\in E(G).
\end{eqnarray}

\begin{claim}\label{uy}
If $uy\in E(G)$, then $\mu(x_1)=1$; if $x_1v_1\in E(G)$, then $y\in B^=$, $\mu(v_1,y)=\mu(y)=1$, $d_B(y)=b(y)$ and $d_{B\setminus D_B}(y)=b(y)-1$.
\end{claim}
\begin{proof}
If $uy\in E(G)$, then $x_1v_1, x_1y\notin E(G)$ as $G$ is $K_4^-$-free. By Claim \ref{exchange-x1v}, $(A_{v_1},B_{v_1})\in\mathscr{P}$, $u\in A_{v_1}^-$ and $d_{A_{v_1}}(u)=a(u)+\mu(u)-1-\mu(u,x_1)$. Note that $y\in D_{B_{v_1}}$ as $d_{B_{v_1}}(y)=d_{B\setminus D_B}(y)\leq b(y)-1$. It follows that $d_{A_{v_1}}(u)=a(u)+\mu(u)-2$ by Claim \ref{degree}, implying $\mu(u,x_1)=1$. The desired result follows by noting that $\mu(x_1)=\mu(u,x_1)$.

If $x_1v_1\in E(G)$, then $uy,x_1y\notin E(G)$ as $G$ is $K_4^-$-free. Clearly, $y\in B^=$, $\mu(y)=\mu(v_1,y)$ and $d_{B\setminus D_B}(y)= b(y)-1$. By Claim \ref{exchange}, $(A_4,B_4):=(A\cup\{y\}\setminus\{u\},B\cup\{u\}\setminus\{y\})\in\mathscr{P}$.
Note that $d_{A_4}(x_1)=d_{A\setminus\{u\}}(x_1)=a(x_1)-1$
and $d_{B_4}(v_1)=d_B(v_1)+\mu(u,v_1)-\mu(v_1,y)\leq b(v_1)+\mu(v_1)-2$.
Thus, $x_1\in D_{A_4}$ and $v_1\in B_4^-$.
By Claim \ref{degree}, $d_{B_4}(v_1)=b(v_1)+\mu(v_1)-2$, indicating $\mu(v_1,y)=1$. Thus $\mu(y)=\mu(v_1,y)=1$, $d_B(y)=b(y)$ and $d_{B\setminus D_B}(y)= b(y)-1$.
\end{proof}

Now, we may further assume that
\begin{eqnarray}\label{DB}
|D_B|\geq2.
\end{eqnarray}
Otherwise, $D_B=\{v_1\}$ as $v_1\in D_B$. If $uy\in E(G)$, then $u\in A_{v_1}^-$ and $x_1,y\in D_{B_{v_1}}$ by Claim \ref{exchange-x1v} and the proof of Claim \ref{uy}. Thus, $A_{v_1}^-=\{u\}$ by Claim \ref{min} and $|D_{B_{v_1}}|\geq2$. If $x_1v_1\in E(G)$, then $v_1\in B_4^-$ and $x_1,y\in D_{A_4}$ by the proof of Claim \ref{uy}. Again, $B_4^-=\{v_1\}$ by Claim \ref{min} and $|D_{A_4}|\geq2$. Thus, we can reduce both cases to \eqref{DB}, as desired.


Let $D=D_B\cup\{y\}$. It follows from \eqref{x1v1uy} and \eqref{DB} that $N_D(v)=\emptyset$ for each $v\in D_B\setminus\{v_1\}$ as $G$ is $\{K_4^-,C_5^+\}$-free and $D_B$ is independent.
This implies that $d_{B\setminus D}(v)=d_B(v)=b(v)-1\geq1$, i.e., $B\setminus D\neq\emptyset$.
By Claim \ref{containment}, $B\setminus D$ is $(b-1)$-degenerate.
Thus there exists $z\in B\setminus D$ such that $d_{B\setminus D}(z)\leq b(z)-1$.
This together with $d_B(z)\geq b(z)$ gives that $N_D(z)\neq\emptyset$
and
\begin{eqnarray}\label{dBz}
d_B(z)=d_{B\setminus D}(z)+d_D(z)\leq b(z)-1+\sum_{x\in N_D(z)}\mu(x,z).
\end{eqnarray}
In what follows, we proceeds our proof by considering $N_D(z)$ according to \eqref{x1v1uy}.

\textbf{Case 1.} $x_1v_1\in E(G)$. By Claim \ref{uy}, we have
$y\in B^=$, $\mu(y)=1$, $d_B(y)=b(y)$ and $d_{B\setminus D_B}(y)=b(y)-1$.
We first establish the following easy but useful claim.
\begin{claim}\label{B=}
$(\mathrm{i})$ There exists $w\in N_{A^=}(x_1)$ such that
$uw\notin E(G)$, $\mu(x_1,w)=\mu(w)$ and $d_{A\setminus \{u,x_1\}}(w)=a(w)-1$.
$(\mathrm{ii})$ If there exists $y'\in N_{B^=}(y)$, then $v_1y'\in E(G)$.
\end{claim}
\begin{proof}
$(\mathrm{i})$ Let $U=\{u,x_1\}$. Clearly, $A\setminus U\neq\emptyset$ as $d_{A\setminus U}(x_1)=d_{A\setminus\{u\}}(x_1)=a(x_1)-1\geq1$. By Claim \ref{containment}, $A\setminus U$ is $(a-1)$-degenerate, implying that there exists $w\in A\setminus U$ such that $d_{A\setminus U}(w)\leq a(w)-1$. It follows that $d_U(w)=d_A(w)-d_{A\setminus U}(w)\geq a(w)+\mu(w)-1-(a(w)-1)=\mu(w)\geq1$, i.e., $N_U(w)\neq\emptyset$. Thus $|N_U(w)|=1$ as $G$ is $K_4^-$-free, implying $d_U(w)\leq\mu(w)$. Then $d_U(w)=\mu(w)$, $d_A(w)=a(w)+\mu(w)-1$ and $d_{A\setminus U}(w)=a(w)-1$. Since $w\in A^=$ and $N_{A^=}(u)=\{x_1\}$ by Claim \ref{diagonal}, we have $uw\notin E(G)$, $x_1w\in E(G)$ and $\mu(x_1,w)=\mu(w)$.

$(\mathrm{ii})$ Suppose that $y'\in N_{B^=}(y)$ such that $v_1y'\notin E(G)$. Since $G$ is $\{K_4^-,C_5^+\}$-free, we have $x_1y,uy,uy'\notin E(G)$.
By Claim \ref{degree}, we have $(A_5,B_5):=(A\cup\{v_1\}\setminus\{u\},B\cup\{u\}\setminus\{v_1\})\in\mathscr{P}$ together with the following formulas: (i) $d_{A_5}(v_1)=d_A(v_1)-\mu(u,v_1)=a(v_1)+\mu(v_1)-2$;
(ii) $d_{B_5}(u)=d_B(u)-\mu(u,v_1)\leq b(u)+\mu(u)-2$;
(iii) $d_{A_5}(x_1)=d_A(x_1)+\mu(v_1,x_1)-\mu(u,x_1)\leq a(x_1)+\mu(x_1)-1$;
(iv) $d_{B_5}(y)=d_B(y)-\mu(v_1,y)=b(y)-1$;
(v) $d_{B_5}(y')=d_B(y')=b(y')+\mu(y')-1$.
It follows that $v_1\in A_5^-$, $u\in B_5^-$, $x_1\in A_5^-\cup A_5^=$, $y\in D_{B_5}\subseteq B_5^-$ and $y'\in B_5^=$. By Claim \ref{min}, $A_5^-=\{v_1\}$, implying $x_1\in A_5^=$.
Thus $x_1v_1yy'$ forms a special path with respect to $(A_5,B_5)$. By Claim \ref{diagonal}, either $x_1y\in E(G)$ or $v_1y'\in E(G)$ as $y\in D_{B_5}$, a contradiction.
\end{proof}

Now, we consider $N_D(z)$ and assert that $v_1\notin N_D(z)$. Otherwise, let $v_1z\in E(G)$. Clearly, $uw,uy,uz,wy,x_1y,wv_1,x_1z\notin E(G)$ and $N_{D_B}(z)=\{v_1\}$ as $G$ is $\{K_4^-,C_5^+\}$-free.
We focus on the partition $(A_4,B_4)=(A\cup\{y\}\setminus\{u\},B\cup\{u\}\setminus\{y\})\in\mathscr{P}$ defined in the second part of the proof of Claim \ref{uy}. Clearly, $x_1,y\in D_{A_4}\subseteq A_4^-$, $v_1\in B_4^-$ and $w\in A_4^=$ as $d_{A_4}(w)=d_A(w)=a(w)+\mu(w)-1$.
Note that $d_{B_4}(z)=d_B(z)-\mu(y,z)\leq b(z)-1+\sum_{x\in N_{D_B}(z)}\mu(x,z)$ by \eqref{dBz}.
It follows that $z\in B_4^=$ as $N_{D_B}(z)=\{v_1\}$ and $z\notin B_4^-$ by Claim \ref{min}.
Then $wx_1v_1z$ forms a special path with respect to $(A_4,B_4)$. By Claim \ref{diagonal}, either $wv_1\in E(G)$ or $x_1z\in E(G)$ as $x_1\in D_{A_4}$, a contradiction.
We further show that there exists $v\in D_B\setminus\{v_1\}$ such that $v\in N_D(z)$. Otherwise, $N_D(z)=\{y\}$. In view of \eqref{dBz}, we know $z\in B^-\cup B^=$. If $z\in B^-$, then $\{u,v_1,x_1,y,z\}$ contains a $C_5^+$ as $uz\in E(G)$ by Claim \ref{complete}. Thus, $z\in N_{B^=}(y)$, implying $v_1\in N_D(z)$ by Claim \ref{B=}(ii), a contradiction.
\begin{claim}\label{configuration}
$N_D(z)=\{v,y\}$ with $\mu(z)=1$ and $d_B(z)=b(z)+1$.
\end{claim}
\begin{proof}
Note that $1\leq|N_{D_B}(z)|\leq2$ as $G$ is $K_{2,3}$-free.  Note that $x_1v,x_1y,x_1z,wv,v_1v,vy\notin E(G)$ as $G$ is $\{K_4^-,C_5^+\}$-free.
By Claim \ref{exchange-x1v}, $\mu(v)=\mu(u,v)=1$ and $(A_{v},B_{v})\in\mathscr{P}$; moreover, $u\in A_{v}^-$ and $x_1\in B_{v}^-$.
Note also that $d_{A_{v}}(w)=d_A(w)-\mu(x_1,w)=d_{A\setminus\{u,x_1\}}(w)=a(w)-1$. Thus $u,w\in A_{v}^-$ and $x_1\in B_{v}^-$, implying $B_{v}^-=\{x_1\}$ by Claim \ref{min}.
If $|N_{D_B}(z)|=2$, then there exists $v'\in D_B\setminus\{v_1,v\}$ such that $x_1v',vv'\notin E(G)$ as $G$ is $K_4^-$-free. Note that $d_{B_{v}}(v')=d_B(v')=b(v')-1$, indicating $v'\in D_{B_{v}}\subseteq B_{v}^-$, a contradiction. Hence, $N_{D_B}(z)=\{v\}$. This implies that $1\leq|N_D(z)|\leq2$.
If $|N_D(z)|=1$, then $d_{B_{v}}(z)=d_B(z)-\mu(v,z)=d_{B\setminus D}(z)\leq b(z)-1$, thus $z\in D_{B_{v}}\subseteq B_{v}^-$, a contradiction.
Thus, we conclude that $N_D(z)=\{v,y\}$.
Observe that $z\in B\setminus B^-$; otherwise, $\{u,v_1,x_1,y,z\}$ contains a $C_5^+$ as $uz\in E(G)$ by Claim \ref{complete}.
Note that $\mu(v)=\mu(y)=1$ by Claims \ref{exchange-x1v} and \ref{uy} as $x_1v, uy\notin E(G)$. 
Hence, $b(z)+\mu(z)-1\leq d_B(z)\leq b(z)+1$ by \eqref{dBz}, giving that $\mu(z)\leq2$. If $\mu(z)=2$, then $d_B(z)=b(z)+1$ and $z\in B^=$.
It follows that $z\in N_{B^=}(y)$, implying $v_1z\in E(G)$ by Claim \ref{B=}(ii), a contradiction.
Hence, $\mu(z)=1$ and $z\notin B^=$, indicating $d_B(z)=b(z)+1$.
\end{proof}

\begin{figure}[ht]
\centering
\begin{picture}(-35.1,123)
\linethickness{0.1pt}
\put(-220,90){\circle*{3.0}\color[rgb]{0.00,0.00,0.00}\circle*{3.0}}
\put(-230,92){$u$}
\put(-180,103){\circle*{3.0}\color[rgb]{0.00,0.00,0.00}\circle*{3.0}}
\put(-184,110){$v$}
\put(-180,77){\circle*{3.0}\color[rgb]{0.00,0.00,0.00}\circle*{3.0}}
\put(-185,83){$v_1$}
\put(-220,50){\circle*{3.0}\color[rgb]{0.00,0.00,0.00}\circle*{3.0}}
\put(-223,39){$x_1$}
\put(-243,50){\circle*{3.0}\color[rgb]{0.00,0.00,0.00}\circle*{3.0}}
\put(-247,39){$w$}
\put(-180,50){\circle*{3.0}\color[rgb]{0.00,0.00,0.00}\circle*{3.0}}
\put(-183,39){$y$}
\put(-145,38){\circle*{3.0}\color[rgb]{0.00,0.00,0.00}\circle*{3.0}}
\put(-143,29){$z$}

\linethickness{0.25pt}
\polygon(-232,78)(-208,78)(-208,102)(-232,102)
\put(-252,88){\small $A^-$}

\linethickness{0.25pt}
\polygon(-252,33)(-208,33)(-208,62)(-252,62)
\put(-235,19){\small $A^=$}

\linethickness{0.25pt}
\polygon(-192,68)(-168,68)(-168,120)(-192,120)
\put(-162,96){\small $D_B$}

\linethickness{0.25pt}
\polygon(-192,33)(-168,33)(-168,62)(-192,62)
\put(-186,19){\small $B^=$}

\put(-214,-3){\bf $(A,B)$}

\put(-50,77){\circle*{3.0}\color[rgb]{0.00,0.00,0.00}\circle*{3.0}}
\put(-55,83){$u$}
\put(-50,103){\circle*{3.0}\color[rgb]{0.00,0.00,0.00}\circle*{3.0}}
\put(-55,109){$u'$}
\put(-10,90){\circle*{3.0}\color[rgb]{0.00,0.00,0.00}\circle*{3.0}}
\put(-4,89){$x_1$}
\put(-50,50){\circle*{3.0}\color[rgb]{0.00,0.00,0.00}\circle*{3.0}}
\put(-54,38){$v$}
\put(-85,38){\circle*{3.0}\color[rgb]{0.00,0.00,0.00}\circle*{3.0}}
\put(-88,24){$z'$}
\put(10,50){\circle*{3.0}\color[rgb]{0.00,0.00,0.00}\circle*{3.0}}
\put(9,39){$v_1$}
\put(-5,50){\circle*{3.0}\color[rgb]{0.00,0.00,0.00}\circle*{3.0}}
\put(-8,39){$y$}
\put(-20,50){\circle*{3.0}\color[rgb]{0.00,0.00,0.00}\circle*{3.0}}
\put(-24,39){$z$}

\linethickness{0.25pt}
\polygon(-62,68)(-38,68)(-38,120)(-62,120)
\put(-88,98){\small $D_{A_v}$}

\linethickness{0.25pt}
\polygon(-62,33)(-38,33)(-38,62)(-62,62)
\put(-56,18){\small $A_v^=$}

\linethickness{0.25pt}
\polygon(-20,78)(10,78)(10,102)(-20,102)
\put(16,88){\small $B_v^-$}

\linethickness{0.25pt}
\polygon(-27,33)(22,33)(22,62)(-27,62)
\put(-11,18){\small $B_v^=$}

\put(-55,-3){\bf $(A_v,B_v)$}

\put(125,77){\circle*{3.0}\color[rgb]{0.00,0.00,0.00}\circle*{3.0}}
\put(120,83){$y$}
\put(125,103){\circle*{3.0}\color[rgb]{0.00,0.00,0.00}\circle*{3.0}}
\put(120,109){$u$}
\put(165,90){\circle*{3.0}\color[rgb]{0.00,0.00,0.00}\circle*{3.0}}
\put(171,89){$v_1$}
\put(125,35){\circle*{3.0}\color[rgb]{0.00,0.00,0.00}\circle*{3.0}}
\put(121,21){$z'$}
\put(100,50){\circle*{3.0}\color[rgb]{0.00,0.00,0.00}\circle*{3.0}}
\put(97,39){$v$}
\put(185,50){\circle*{3.0}\color[rgb]{0.00,0.00,0.00}\circle*{3.0}}
\put(184,39){$x_1$}
\put(155,35){\circle*{3.0}\color[rgb]{0.00,0.00,0.00}\circle*{3.0}}
\put(152,23){$u'$}
\put(155,50){\circle*{3.0}\color[rgb]{0.00,0.00,0.00}\circle*{3.0}}
\put(159,50){$z$}

\linethickness{0.25pt}
\polygon(113,68)(137,68)(137,120)(113,120)
\put(87,98){\small $D_{A_7}$}

\linethickness{0.25pt}
\polygon(93,18)(133,18)(133,57)(93,57)
\put(71,35){\small $A_7^=$}

\linethickness{0.25pt}
\polygon(155,78)(185,78)(185,102)(155,102)
\put(190,88){\small $B_7^-$}

\linethickness{0.25pt}
\polygon(148,18)(197,18)(197,57)(148,57)
\put(202,35){\small $B_7^=$}

\put(120,-3){\bf $(A_7,B_7)$}

\linethickness{0.5pt}



\linethickness{0.7pt}
\Line(-220,90)(-180,77)
\Line(-220,90)(-180,103)
\Line(-220,90)(-220,50)
\Line(-180,77)(-220,50)
\Line(-180,77)(-180,50)
\Line(-180,103)(-145,38)
\Line(-220,50)(-243,50)
\Line(-180,50)(-145,38)

\Line(-10,90)(-50,77)
\Line(-10,90)(-50,103)
\Line(-10,90)(10,50)
\Line(-50,77)(10,50)
\Line(-50,77)(-50,50)
\Line(-50,103)(-85,38)
\Line(10,50)(-5,50)
\Line(-5,50)(-20,50)
\Line(-20,50)(-50,50)
\Line(-50,50)(-85,38)

\Line(125,103)(165,90)
\Line(125,103)(185,50)
\Line(125,103)(100,50)
\Line(125,77)(165,90)
\Line(125,77)(155,50)
\Line(165,90)(185,50)
\Line(185,50)(155,35)
\Line(155,35)(155,50)
\Line(155,35)(125,35)
\Line(155,50)(100,50)
\Line(100,50)(125,35)


\end{picture}
\caption{Partitions in $\mathscr{P}$}
\end{figure}
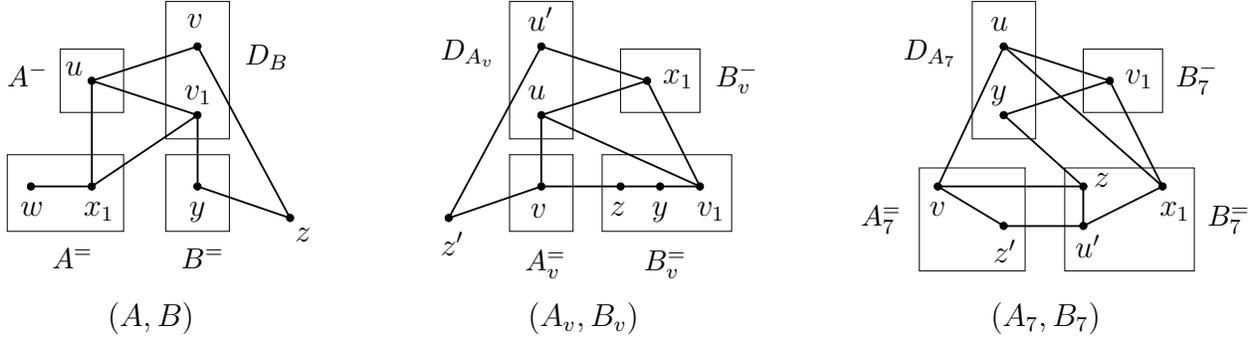


Note that $(A_{v},B_{v})\in\mathscr{P}$ by Claim \ref{exchange-x1v}; additionally, $u\in A_{v}^-$, $v\in A_{v}^=$ and $x_1\in B_{v}^-$. In what follows, we show that $B_v^-=\{x_1\}$, $u,w\in D_{A_v}$,
$v_1\in N_{B_v^=}(x_1)$ with $d_{B_v\setminus\{x_1\}}(v_1)=b(v_1)-1$,
$y\in N_{B_v^=}(v_1)$ with $d_{B_v\setminus\{x_1,v_1\}}(y)=b(y)-1$, and
$v\in N_{A_v^=}(u)$ with $d_{A_v\setminus D_{A_v}}(v)=a(v)-1$.
If so, we may view $B_{v}$, $A_{v}$ as the new parts $A$, $B$ by the symmetry between the functions $a,b$, and make sure that we are still in Case 1 as $v_1u\in E(G)$.

Recall that $\mu(v)=\mu(y)=1$. Since $G$ is  $\{K_4^-,C_5^+\}$-free, we have $x_1v,x_1y,vy,uy\notin E(G)$. Note that $d_{A_{v}}(w)=d_{A\setminus\{u,x_1\}}(w)=a(w)-1$ and
$d_{B_{v}}(v_1)=d_B(v_1)+\mu(x_1,v_1)=b(v_1)-1+\mu(x_1,v_1)\leq b(v_1)+\mu(v_1)-1$.
It follows that $w\in D_{A_{v}}$ and $v_1\in B_{v}^-\cup B_{v}^=$. Since $u,w\in A_{v}^-$ and $x_1\in B_{v}^-$, we have $B_v^-=\{x_1\}$ and $v_1\in B_{v}^=$ by Claim \ref{min}.
Thus $d_{B_{v}}(v_1)=b(v_1)+\mu(v_1)-1$ and $\mu(x_1,v_1)=\mu(v_1)$.
This implies that $d_{B_{v}\setminus\{x_1\}}(v_1)=d_{B_{v}}(v_1)-\mu(x_1,v_1)=b(v_1)-1$ and $d_{B_{v}\setminus\{x_1,v_1\}}(y)=d_B(y)-\mu(v_1,y)=b(y)-1$.
In addition, $N_{A_{v}^-}(v)=\{u\}$ as $G$ is $C_5^+$-free and $d_{A_v\setminus D_{A_v}}(v)=d_{A_v}(v)-\mu(u,v)=a(v)-1$.
It remains to show that $u\in D_{A_v}$. By Claim \ref{containment}, $A_v\setminus D_{A_v}$ is $(a-1)$-degenerate. Thus there exists $w'\in A_v\setminus D_{A_v}$ such that $d_{A_v\setminus D_{A_v}}(w')\leq a(w')-1$ and $|N_{D_{A_v}}(w')|=1$ by Claim \ref{y}. We may assume that $N_{D_{A_v}}(w')=\{u_1\}$ and $u\notin D_{A_v}$.
Clearly, $u_1v_1\notin E(G)$ and $w'\neq u$ as $G$ is $K_4^-$-free.
Now, we may view $B_{v}$, $A_{v}$ as the new parts $A$, $B$ by the symmetry between the functions $a,b$, and $x_1,u_1,v_1$ play roles in $(B_v,A_v)$ as that $u,v,x_1$ in the original partition $(A,B)$, respectively.
Let $A_6=A_v\cup\{v_1\}\setminus\{u_1\}$ and $B_6=B_v\cup\{u_1\}\setminus\{v_1\}$. By Claim \ref{exchange-x1v}, we have $\mu(u_1)=1$, $(A_6,B_6)\in\mathscr{P}$, $v_1\in A_6^-$ and $x_1\in B_6^-$.
Note that $d_{A_6}(w')=d_{A_v\setminus D_{A_v}}(w')\leq a(w')-1$ and $d_{B_6}(y)=d_{B_v}(y)-\mu(v_1,y)=b(y)-1$. Thus, $v_1,w'\in A_6^-$ and $x_1,y\in B_6^-$. This contradicts Claim \ref{min}. Hence, $u\in D_{A_v}$.

Now, we consider the partition $(B_v,A_v)$, which satisfies all the conditions of Case 1 by the above argument. We mention that $x_1, u, v_1, v, y$ play roles in $(B_v,A_v)$ as that $u, v_1, x_1, y, w$ in the original partition $(A,B)$, respectively. By Claim \ref{configuration}, we may assume that there exist $u'\in D_{A_v}\setminus\{u\}$ and $z'\in A_v\setminus(D_{A_v}\cup\{v\})$ such that $N_{D_{A_v}\cup\{u\}}(z')=\{v,u'\}$ , $\mu(u')=\mu(z')=1$ and $d_{A_v}(z')=a(z')+1$.

Let $A_7=A_v\cup\{y\}\setminus\{u'\}$ and $B_7=B_v\cup\{u'\}\setminus\{y\}$. Note that $u'y,u'u,u'v_1,u'v,x_1y,uy,vy\notin E(G)$ as $G$ is  $\{K_4^-,C_5^+\}$-free. Then we have the following equalities:
(i) $d_{A_7}(y)=d_{A_v}(y)=d_G(y)-d_{B_v}(y)=a(y)-1$;
(ii) $d_{A_7}(u)=d_{A_v}(u)=a(u)-1$;
(iii) $d_{A_7}(v)=d_{A_v}(v)=a(v)$;
(iv) $d_{B_7}(u')=d_{B_v}(u')=d_G(u')-d_{A_v}(u')=b(u')$;
(v) $d_{B_7}(x_1)=d_{B_v}(x_1)+\mu(u',x_1)=b(x_1)+\mu(x_1)-1$;
(vi) $d_{B_7}(v_1)=d_{B_v}(v_1)-\mu(v_1,y)=b(v_1)+\mu(v_1)-2$.
We claim that $(A_7,B_7)\in\mathscr{P}$.
Clearly, $A_7$ is $(a-1)$-meager. If not, then there is an $a$-nice subset $A'\subseteq A_7$. Since $A_v$ is $(a-1)$-meager, we have $y\in A'$ and $d_{A_7}(y)\geq d_{A'}(y)\geq a(y)+\mu(y)-1=a(y)$, a contradiction.
Now we prove that $B_7$ is $(b-1)$-meager. If not, then there is a $b$-nice subset $B'\subseteq B_7$. Since $B_v$ is $(b-1)$-meager, we have $u'\in B'$ and $d_{B_7}(u')\geq d_{B'}(u')\geq b(u')+\mu(u')-1=b(u')$. Thus, $d_{B_7}(u')=d_{B'}(u')=b(u')$, implying $x_1\in B'$ as $x_1u\in E(G)$.
Then, $d_{B_7}(x_1)\geq d_{B'}(x_1)\geq b(x_1)+\mu(x_1)-1$. It follows that $d_{B_7}(x_1)=d_{B'}(x_1)=b(x_1)+\mu(x_1)-1$, implying $v_1\in B'$ as $v_1x_1\in E(G)$. Hence, $d_{B_7}(v_1)\geq d_{B'}(v_1)\geq b(v_1)+\mu(v_1)-1$, a contradiction.
Thus, $(A_7,B_7)$ is an $(a-1,b-1)$-meager partition.
By \eqref{weight3} and \eqref{=}, $\omega(A_7,B_7)=\omega(A,B)$.
As claimed.


Note that $u,y\in D_{A_7}$, $v\in A_7^=$, $v_1\in B_7^-$ and $u',x_1\in B_7^=$. In what follows, we prove that $B_7^-=\{v_1\}$,
$x_1\in N_{B_7^=}(v_1)$ with $d_{B_7\setminus\{v_1\}}(x_1)=b(x_1)-1$, and
$v\in N_{A_7^=}(u)$ with $d_{A_7\setminus D_{A_7}}(v)=a(v)-1$,
If so, we may view $B_7$, $A_7$ as the new parts $A$, $B$ by the symmetry between the functions $a,b$, and again we are still in Case 1 as $x_1u\in E(G)$.

By Claim \ref{min}, $B_7^-=\{v_1\}$. Now, we show that $d_{B_7\setminus\{v_1\}}(x_1)=b(x_1)-1$. Note that $d_{B_7\setminus\{v_1\}}(x_1)=d_{B_7}(x_1)-\mu(v_1,x_1)= b(x_1)+\mu(x_1)-1-\mu(v_1,x_1)\geq b(x_1)-1$. It suffices to prove that $d_{B_7\setminus\{v_1\}}(x_1)\leq b(x_1)-1$. Suppose for a contradiction that $d_{B_7\setminus\{v_1\}}(x_1)> b(x_1)$.
By Claim \ref{containment}, $B_7\setminus\{v_1\}$ is $(b-1)$-degenerate as $G$ has no $(a,b)$-feasible pair. This implies that there exists $y''\in B_7\setminus\{v_1\}$ such that $d_{B_7\setminus\{v_1\}}(y'')\leq b(y'')-1$. Clearly, $y''\neq x_1$ and $d_{B_7}(y'')\geq b(y'')+\mu(y'')-1$.
Note also that $d_{B_7}(y'')=d_{B_7\setminus\{v_1\}}(y'')+\mu(v_1,y'')\leq b(y'')-1+\mu(y'')$.
Thus, $d_{B_7}(y'')=b(y'')+\mu(y'')-1$ and $y''\in B_7^=$.
Then $v u v_1 y''$ forms a special path with respect to $(A_7,B_7)$. By Claim \ref{diagonal}, either $v_1v\in E(G)$ or $uy''\in E(G)$ as $u\in D_{A_7}$. In either case, we have a $K_4^-$, a contradiction.
It remains to prove that $d_{A_7\setminus D_{A_7}}(v)=a(v)-1$.
By Claim \ref{complete}, we have $N_{D_{A_7}}(v)=\{u\}$ as $G$ is $C_5^+$-free. Thus, $d_{A_7\setminus D_{A_7}}(v)=d_{A_7}(v)-\mu(u,v)=a(v)-1$ by noting that $\mu(v)=1$. As desired.

Now, we consider the partition $(B_7,A_7)$, and $v_1, u, x_1, v$ play roles in $(B_7,A_7)$ as that $u, v_1, x_1, y$ in the original partition $(A,B)$, respectively. We show that $u'z,uz'\in E(G)$; if so, then $\{u,v,z,u',z'\}$ contains a $C_5^+$, a contradiction. Recall that $\mu(z)=1$ and $d_B(z)=b(z)+1$ by Claim \ref{configuration}. If $u'z\notin E(G)$, then $d_{B_7}(z)=d_{B_v}(z)-\mu(y,z)=d_B(z)-\mu(v,z)-\mu(y,z)=b(z)-1$, implying $z\in D_{B_7}$. Thus $u,y\in A_7^-$ and $v_1,z\in B_7^-$, contradicting Claim \ref{min}.
Next, we show $uz'\in E(G)$. Since $G$ is $K_{2,3}$-free, $yz'\notin E(G)$. Note that $\mu(z')=1$ and $d_{A_v}(z')=a(z')+1$. Thus $d_{A_7}(z')=d_{A_v}(z')-\mu(u',z')=a(z')$, implying $z'\in A_7^=$. By Claim \ref{B=}(ii), $uz'\in E(G)$ as $z'\in N_{A_7^=}(v)$. Thus we complete the proof of Case 1.

\textbf{Case 2.} $uy\in E(G)$. Clearly, $x_1v_1\notin E(G)$ and $N_{D_B}(y)=\{v_1\}$. By Claims \ref{exchange-x1v} and \ref{uy}, $\mu(v_1)=\mu(x_1)=1$.
Note that $1\leq|N_D(z)|\leq2$ as $G$ is $K_{2,3}$-free.
If $|N_D(z)|=2$, then $yz\in E(G)$; otherwise, we have $z\in B\setminus D_B$ such that $d_{B\setminus D_B}(z)\leq b(z)-1$, implying $|N_{D_B}(z)|=1$ by Claim \ref{y}, a contradiction. It follows that $v_1z\notin E(G)$ as $G$ is $K_4^-$-free. Thus there exists $v\in D_B\setminus\{v_1\}$ such that $vz\in E(G)$ and $\{u,v,v_1,y,z\}$ contains a $C_5^+$, a contradiction. Hence, $|N_D(z)|=1$ and $d_B(z)\leq b(z)-1+\mu(z)$ by \eqref{dBz}.
\begin{claim}\label{NDz-v2}
$N_D(z)=\{v_2\}$ for some $v_2\in D_B\setminus\{v_1\}$.
\end{claim}
\begin{proof}
Suppose not. Clearly, $z\in B^=$ as $G$ is $K_4^-$-free.
It follows that $d_{B\setminus D}(z)=b(z)-1$ and $d_D(z)=\mu(z)$.
If $N_D(z)=\{v_1\}$, then $x_1uv_1z$ forms a special path with respect to $(A,B)$. Since $v_1\in D_B$, either $x_1v_1\in E(G)$ or $uz\in E(G)$ by Claim \ref{diagonal}, implying a $K_4^-$ in both cases, a contradiction.
If $N_D(z)=\{y\}$, then $d_B(z)=b(z)+\mu(z)-1$ and $\mu(y,z)=\mu(z)$.
Since $G$ is $\{K_4^-,C_5^+\}$-free, we have $x_1v_1,x_1y,x_1z,v_1z\notin E(G)$.
By Claim \ref{exchange-x1v}, $(A_{v_1},B_{v_1})\in\mathscr{P}$, $u\in A_{v_1}^-$, $v_1\in A_{v_1}^=$ and $x_1\in D_{B_{v_1}}\subseteq B_{v_1}^-$.
Note that $d_{B_{v_1}}(y)=d_B(y)-\mu(v_1,y)=d_{B\setminus D_B}(y)\leq b(y)-1$. It follows that $y\in D_{B_{v_1}}\subseteq B_{v_1}^-$. Thus, $A_{v_1}^-=\{u\}$ by Claim \ref{min}.
Since $G$ is $C_5^+$-free, we have $N_{D_{B_{v_1}}}(z)=\{y\}$. Thus $d_{B_{v_1}\setminus D_{B_{v_1}}}(z)=d_{B_{v_1}}(z)-\mu(y,z)=d_B(z)-\mu(y,z)=b(z)-1$.
Moreover, $v_1\in A_{v_1}^=$ with $d_{A_{v_1}\setminus\{u\}}(v_1)=d_{A_{v_1}}(v_1)-\mu(u,v_1)=a(v_1)-1$.
Now, we view $A_{v_1}$ $B_{v_1}$ as the new parts $A$, $B$ and the case can be reduced to Case 1 as $v_1y\in E(G)$.
In fact, $v_1,u,y,z$ play roles in $(A_{v_1},B_{v_1})$ as that $x_1,u,v_1,y$ in the partition $(A,B)$ of Case 1, respectively.
\end{proof}

Let $Z:=\{z^*\in B\setminus D:d_{B\setminus D}(z^*)\leq b(z^*)-1\}$. Clearly, $z\in Z\subseteq B^-\cup B^=$. By Claim \ref{NDz-v2}, for each $z^*\in Z$, we may assume that $N_D(z^*)=\{v^*\}$ for some $v^*\in D_B\setminus\{v_1\}$. Now, we show $uz^*\in E(G)$ for each $z^*\in Z$. If $z^*\in B^-$, then we're done by Claim \ref{complete}. Thus, $z^*\in B^=$ and $x_1uv^*z^*$ forms a special path with respect to $(A,B)$. By Claim \ref{diagonal}, either $x_1v^*\in E(G)$ or $uz^*\in E(G)$. If $x_1v^*\in E(G)$, then the case can be reduced to Case 1, where $z^*$ and $v^*$ play the roles of $y$ and $v_1$. Thus we conclude that $uz^*\in E(G)$ for each $z^*\in Z$.

Note that $N_{D\cup Z}(y)=N_{D_B}(y)$ as $yz^*\notin E(G)$ for each $z^*\in Z$. Thus $d_{B\setminus(D\cup Z)}(y)=d_{B\setminus D_B}(y)=b(y)-1\geq1$, i.e., $B\setminus(D\cup Z)\neq\emptyset$. By Claim \ref{containment}, $B\setminus(D\cup Z)$ is $(b-1)$-degenerate. Hence, there exists $z'\in B\setminus(D\cup Z)$ such that $d_{B\setminus(D\cup Z)}(z')\leq b(z')-1$, implying $|N_{D\cup Z}(z')|\geq1$ by noting that $d_B(z')\geq b(z')$.
Since $u$ is adjacent to each vertex in $D\cup Z$, we have $|N_{D\cup Z}(z')|\leq2$ as $G$ is $K_{2,3}$-free.
If $|N_{D\cup Z}(z')|=2$, then $N_{D\cup Z}(z')\nsubseteq D_B$ by Claim \ref{y}. It is easy to check that $G$ contains a $K_4^-$ or $C_5^+$, a contradiction.
Let $N_{D\cup Z}(z')=\{y'\}$.
If $y'\in D$, then $d_{B\setminus D}(z')=d_{B\setminus(D\cup Z)}(z')\leq b(z')-1$, indicating $z'\in Z$, a contradiction.
Thus $y'\in Z$ and $d_{B\setminus(D_B\cup\{y'\})}(z')=d_{B\setminus(D\cup Z)}(z')\leq b(z')-1$.
Now, we may view $y'$, $z'$ and $D_B\cup\{y'\}$ as the new $y$, $z$ and $D$, respectively. Since $uy'\in E(G)$, we are still in Case 2. By Claim \ref{NDz-v2}, we have $N_{D_B\cup\{y'\}}(z')\subseteq D_B$. This leads to a contradiction as $y'\notin D_B$, completing the proof of Case 2. Thus, we complete the proof of Theorem \ref{multigraph}.

\section*{Acknowledgements}
The authors would like to thank Jie Ma for helpful comments and discussions.

\end{document}